\documentclass[10pt]{amsart}

\usepackage[margin=1in]{geometry}
\usepackage{amsmath}
\usepackage{amsxtra}
\usepackage{amstext}
\usepackage{amssymb}

\newtheorem{definition}{Definition}

\newtheorem{example}{Example}

\newtheorem{theorem}{Theorem}[section]
\newtheorem{lemma}[theorem]{Lemma}

\newcommand{\expect}[1]{\mathbb{E}\left[#1\right]}

\newcommand{\indicator}[1]{\mathbb{I}_{#1}}

\title{Properties of Conditional Expectation Operators and Sufficient Subfields}

\author{Andrew Tausz}
\address{Stanford University, Stanford, CA, 94305}
\email{atausz@stanford.edu}

\date{May 13, 2010}

\begin{document}
\maketitle

\begin{abstract}
We discuss some properties of conditional expectation operators, and use these facts to prove an interesting counterexample regarding sufficient statistics. In particular, we show that there exists sufficient random variables $X$ and $Y$, such that $(X, Y)$ are jointly not sufficient. We follow the work in \cite{1962} and \cite{1961}, presenting complete proofs and supplying missing steps. 


\end{abstract}

\section{Introduction}

For now, we fix a probability space $(\Omega, \mathcal{F}, \mathbb{P})$. Let $X \in L_1(\Omega, \mathcal{F}, \mathbb{P})$ be an integrable random variable. Then, if $\mathcal{G}$ is a sub-$\sigma$-field (which we simply call a subfield) of $\mathcal{F}$, the conditional expectation of $X$ with respect to $\mathcal{G}$ is defined to be the random variable $\expect{X | \mathcal{G}}$ that satisfies:

\begin{enumerate}
\item For all $B \in \mathcal{G}$, $\expect{\expect{X | \mathcal{G}} \indicator{B}}  = \expect{X \indicator{B}}$
\item $\expect{X | \mathcal{G}}$ is $\mathcal{G}$-measurable
\end{enumerate}

The existence of the conditional expectation is a standard result that can be found in \cite{schervish}. We say that an operator $T$ on $L_1$ is a conditional expectation operator if $T = \expect{\cdotp | \mathcal{G}}$ for some sub-sigma field $\mathcal{G} \subset \mathcal{F}$. Let $\{T_n\}$ be a sequence of conditional expectation operators and write the sequence of iterates as $S_n = T_n \cdots T_2 T_1$. In this report we will discuss certain properties of the iterates of a sequence of conditional expectation operators. 

Before diving into the material, a notational explanation is in order. Unfortunately in the world of probability, capital letters are used to denote random variables, but in the world of functional analysis lower case letters seem to be popular for elements of function spaces. Although this paper uses ideas from both fields, our ultimate goal is statistical, so we will use capital letters. The letters $T$ and $Q$ are reserved for operators on a Hilbert space, and $X$ is always an element of the function space in question. 

An important interpretation of the conditional expectation is that it is an orthogonal projection operator onto the space of $\mathcal{G}$-measurable functions within the Hilbert space $L_2(\Omega, \mathcal{F}, \mathbb{P})$. This perspective will be the driving force behind most of the ideas here. For completeness we include the following result.

\begin{theorem}
Let $T$ be the operator on the Hilbert space $L_2(\Omega, \mathcal{F}, \mathbb{P})$ defined by $T(X) = \expect{X | \mathcal{G}}$. Then $T$ has the following properties.
\begin{enumerate}
\item $T$ is self-adjoint: $T^{*} = T$
\item $T$ is idempotent: $T^2 = T$
\item $I - T$ is the projection onto $\left(L_2(\Omega, \mathcal{G}, \mathbb{P})\right)^{\perp}$
\item $T$ is a contraction: $||T|| \leq 1$
\end{enumerate}
\end{theorem}

\begin{proof}
These results follow from the basic properties of conditional expectations. See \cite{Shorack}.
\end{proof}

\section{Some Preparatory Lemmas}

We begin by stating and proving a lemma found in \cite{1961}. 

\begin{definition}
A sequence $\{a_n\}_1^{\infty}$ of real numbers is said to be convex if $a_{i+1} \leq \frac{1}{2}(a_i + a_{i+2})$ for all $i \geq 1$.
\end{definition}

\begin{lemma}
Let $\{a_n\}_1^{\infty}$ be a convex and bounded sequence of real numbers. Then 
$$\sum_{n=1}^{\infty} n \Delta^2 a_n = a_1 - \lim_{n \rightarrow \infty} a_n$$
\end{lemma}

\begin{proof}
We note that $\Delta^2 a_n = a_{n+2} - 2 a_{n+1} + a_n$ is non-negative, by the assumption of convexity. Thus we can rearrange the sum as we wish. In particular, we are free to use Fubini's theorem.

\begin{eqnarray*}
\sum_{n=1}^{\infty} n \Delta^2 a_n & = & \sum_{n=1}^{\infty} \sum_{k=1}^{n} \Delta^2 a_n\\
 & = & \sum_{k=1}^{\infty} \sum_{n=k}^{\infty} \Delta^2 a_n\\
 & = & \sum_{k=1}^{\infty} \sum_{n=k}^{\infty} \Delta a_{n+1} - \Delta a_{n}\\
 & = & \sum_{k=1}^{\infty} \Delta a_{k} + \lim_{n \rightarrow \infty} \Delta a_{n+1}\\
 & = & a_{1} - \lim_{n \rightarrow \infty} a_n
\end{eqnarray*}

\end{proof}

\begin{lemma}
Let $\{a_n\}_{n=0}^{\infty}$ be a sequence of complex numbers such that $c^2 = \sum_{n=1}^{\infty}n|a_n - a_{n+1}|^2 < \infty$. Define 
$$b_n = \sum_{k=1}^{2^n} \frac{a_k}{2^n}$$
Then
\begin{enumerate}
\item $\sup_{n \geq 1} |a_n| \leq 3 \left( \sup_{n \geq 0} |b_n| \right) + |c|$
\item If $\lim_{n \rightarrow \infty} b_n = a_0$, then $\lim_{n \rightarrow \infty} a_n = a_0$
\end{enumerate}
\end{lemma}

\begin{proof}
Suppose without loss of generality that $a_0 = 0$. If $a_0 \neq 0$, then we can simply subtract $a_0$ from each term and the proof would still work. Define
$$c_{n}^2 = \sum_{k=2^n}^{\infty} k|a_k - a_{k+1}|^2$$
Then we have that
\begin{eqnarray*}
\max_{2^n \leq j \leq k \leq 2^{n+1}} |a_j - a_k| & \leq & \sum_{k=2^n}^{2^{n+1}-1} |a_k - a_{k+1}|\\
 & \leq & \left(2^n \sum_{k=2^n}^{2^{n+1}-1} |a_k - a_{k+1}| \right)^{\frac{1}{2}}\\
 & \leq & \left(\sum_{k=2^n}^{2^{n+1}-1} k |a_k - a_{k+1}| \right)^{\frac{1}{2}}\\
 & \leq & \left(\sum_{k=2^n}^{\infty} k |a_k - a_{k+1}| \right)^{\frac{1}{2}}\\
 & = & |c_n|
\end{eqnarray*}
Let $m$ be the integer that satisfies $2^m \leq n \leq 2^{n+1}$. Then we have that
\begin{eqnarray*}
|a_n| & = & |b_m - 2b_{m+1} + \sum_{k=2^m+1}^{2^{m+1}} 2^{-m} (a_k - a_n)|\\
 & \leq & |b_m| + 2|b_{m+1}| + |c_m|
\end{eqnarray*}
Taking supremums of both sides of the above, we get that
$$\sup_{n \geq 1} |a_n| \leq 3 \left( \sup_{n \geq 0} |b_n| \right) + |c|$$
To get the second result, suppose that $\lim_{n \rightarrow \infty} b_n = a_0 = 0$. Then we have that
\begin{eqnarray*}
\lim_{n \rightarrow \infty} a_n & = & \lim_{m \rightarrow \infty} \left( b_m - 2b_{m+1} + \sum_{k=2^m+1}^{2^{m+1}} 2^{-m} (a_k - a_n) \right)\\
 & = & \lim_{m \rightarrow \infty} b_m - \lim_{m \rightarrow \infty} 2b_{m+1} + \lim_{m \rightarrow \infty} \sum_{k=2^m+1}^{2^{m+1}} 2^{-m} (a_k - a_n)\\
 & = & 0
\end{eqnarray*}
\end{proof}

\begin{lemma}
Let $T$ be a linear self-adjoint operator with norm $||T|| \leq 1$ on a Hilbert space $\mathcal{H}$. Then $\forall X \in \mathcal{H}$,
$$\sum_{n=1}^{\infty} ||T^n(X) - T^{n+2}(X)||^2 \leq ||X||^2$$
Furthermore, the $L_2$-limit of $T^{2n}$ exists and is an orthogonal projection operator.
\end{lemma}

\begin{proof}
Fix some $X \in \mathcal{H}$ and define $a_n = ||T^n(X)||^2$. Since $||T|| \leq 1$, we have that the sequence $\{a_n\}_1^{\infty}$ is bounded. This sequence is also convex, since (by using self-adjointness of $T$)
\begin{eqnarray*}
\Delta^2 a_n & = & a_{n+2} - 2 a_{n+1} + a_n \\
 & = & ||T^{n+2}(X)||^2 - 2 ||T^{n+1}(X)||^2 + ||T^{n}(X)||^2 \\
 & = & ||T^{n+2}(X)||^2 - 2 \langle T^{n+1}(X), T^{n+1}(X) \rangle + ||T^{n}(X)||^2 \\
 & = & ||T^{n+2}(X)||^2 - 2 \langle T^{n}(X), T^{n+2}(X) \rangle + ||T^{n}(X)||^2 \\
 & = & \langle T^{n}(X) - T^{n+2}(X), T^{n}(X) - T^{n+2}(X) \rangle \\
 & = & ||T^{n}(X) - T^{n+2}(X)||^2 \\
 & \geq & 0
\end{eqnarray*}
Thus by applying Lemma 1, we get that
$$\sum_{n=1}^{\infty} n ||T^{n}(X) - T^{n+2}(X)||^2 = ||T(X)||^2 - \lim_{n \rightarrow \infty} ||T^{n}(X)||^2$$
However, since $||T|| \leq 1$, we have that $||T(X)|| \leq ||X||$. By using this, and positivity of the limit, we get that
$$\sum_{n=1}^{\infty} ||T^n(X) - T^{n+2}(X)||^2 \leq ||X||^2$$
To show convergence of $\{T^{2n}\}$, we show that it is a Cauchy sequence.
\begin{eqnarray*}
\lim_{n, m \rightarrow \infty} ||T^{2n}(X) - T^{2m}(X)||^2 & = & \lim_{n, m \rightarrow \infty} \langle T^{2n}(X) - T^{2m}(X), T^{2n}(X) - T^{2m}(X) \rangle\\
 & = & \lim_{n, m \rightarrow \infty} ||T^{2n}(X)||^2 - 2 \langle T^{2n}(X), T^{2m}(X) \rangle + ||T^{2m}(X)||^2
\end{eqnarray*}
But we also have that by using the self-adjointness of $T$ we get
\begin{eqnarray*}
a_{n+m} & = & ||T^{n+m}(X)||^2\\
 & = & \langle T^{n+m}(X), T^{n+m}(X) \rangle\\
 & = & \langle T^{m}(X), T^{2n+m}(X) \rangle\\
 & = & \langle T^{2m}(X), T^{2n}(X) \rangle
\end{eqnarray*}
So we have that
\begin{eqnarray*}
\lim_{n, m \rightarrow \infty} ||T^{2n}(X) - T^{2m}(X)||^2 & = & \lim_{n, m \rightarrow \infty} (a_{2n} - 2 a_{n+m} + a_{2m})\\
 & = & \lim_{n, m \rightarrow \infty} (a_{2n} - a_{n+m}) + (a_{2m} - a_{n+m})\\
 & = & 0
\end{eqnarray*}

By the definition of completeness in a Hilbert space, we have that the sequence $\{T^{2n}\}$ converges in the $L_2$-norm to some $Q \in L_2$. Since $T$ is self-adjoint, so must $Q$. To obtain idempotence, note that since $\{T^{4n}\}$ is a subsequence of $\{T^{2n}\}$, it must also converge to $Q$. So for all $X \in \mathcal{H}$, we have that
\begin{eqnarray*}
||Q^2(X) - Q(X)|| & \leq & ||Q^2(X) - T^{4n}(X)|| + ||T^{4n}(X) - T^{2n}(X)|| + ||T^{2n}(X) - Q(X)||
\end{eqnarray*}
Since each of the quantities on the right side converge to 0, $Q$ is idempotent and is hence an orthogonal projection.
\end{proof}

\begin{theorem}
Let $T$ be a linear self-adjoint operator on $L_2 = L_2(\Omega, \mathcal{F}, \mathbb{P})$ with norm $||T|| \leq 1$. Suppose that $X \in L_2$ and that the following limit exists almost surely
\begin{equation}
\lim_{n \rightarrow \infty} \sum_{k=1}^{2^n} 2^{-n} T^{2k}(X)
\label{property_x}
\end{equation}
Then $\lim_{n \rightarrow \infty} T^{2n}(X) = Q(X)$.
\end{theorem}

\begin{proof}
To prove this, we will use Lemma 3.2. Define $a_n(X) = T^{2n}(X)$. Then to invoke Lemma 3.2 we must show that $c^2 = \sum_{n=1}^{\infty}n|a_n - a_{n+1}|^2 < \infty$ a.s. However, we have that
\begin{eqnarray*}
\int_{\Omega} \sum_{n=1}^{\infty} n |T^{2n}(X) - T^{2(n+1)}(X)|^2 d\mathbb{P} & = & \sum_{n=1}^{\infty} n \int_{\Omega} |T^{2n}(X) - T^{2(n+1)}(X)|^2 d\mathbb{P}\\
 & = & \sum_{n=1}^{\infty} n ||T^{2n}(X) - T^{2(n+1)}(X)||^2\\
 & = & \frac{1}{2} \sum_{n=1}^{\infty} 2n ||T^{2n}(X) - T^{2(n+1)}(X)||^2\\
 & \leq & \frac{1}{2} ||X||^2\\
 & < & \infty
\end{eqnarray*}
Above we used the posititivity of the integrand to swap the sum and the integral, and we applied Lemma 3.3 to the operator $T^2$. Thus $\sum_{n=1}^{\infty} n |T^{2n}(X) - T^{2(n+1)}(X)|^2$ is finite almost surely. Thus by Lemma 3.2 we get the desired conclusion.
\end{proof}

\begin{theorem}
Let $T$ be a linear self-adjoint operator on $L_2 = L_2(\Omega, \mathcal{F}, \mathbb{P})$ that satisfies the following two conditions
\begin{enumerate}
\item $||T(X)||_1 \leq ||X||_1$
\item $||T(X)||_{\infty} \leq ||X||_{\infty}$
\end{enumerate}
Then for all $x \in L_2$ we have that $\lim_{n \rightarrow \infty} T^{2n}(X) = Q(X)$.
\end{theorem}

\begin{proof}
To prove this, we invoke the Riesz-Thorin theorem which can be found in \cite{folland} on the interpolation $L_1 \cap L_{\infty} \subset L_2$. In particular we get that $T$ is a bounded operator on $L_2$ and that $||T||_2 \leq \max\left(||T||_1, ||T||_{\infty}\right) = 1$. Furthermore, we get the appropriate condition (\ref{property_x}) by chapter 3 of \cite{dunford}. Thus we conclude by Theroem 3.4 that $\lim_{n \rightarrow \infty} T^{2n}(X) = Q(X)$.
\end{proof}

\section{A Result about Conditional Expectation Operators}

Now that we have the basic lemmas out of the way, we can prove a result about conditional expectation operators. We will use this in our discussion about sufficiency in the next section. We note that the authors in \cite{1961} and \cite{1962} define a conditional expectation operator $\expect{\cdotp|\mathcal{G}}$ to have range $\bar{\mathcal{G}}$ (the completion of the $\sigma$-field $\mathcal{G}$). 

\begin{theorem}
Let $T_1$ and $T_2$ be conditional expectation operators on $(\Omega, \mathcal{F}, \mathbb{P})$. In other words, let $T_1 = \expect{\cdotp | \mathcal{G}_1}$ and $T_2 = \expect{\cdotp | \mathcal{G}_2}$, where $\mathcal{G}_1$ and $\mathcal{G}_2$ are subfields of $\mathcal{F}$. Let $T_{2k-1} = T_1$ and $T_{2k} = T_2$ extend the sequence for all positive integers. Then, we have that for all $X \in L_2$,
$$\lim_{n \rightarrow \infty} S_n(X) = \expect{X|\bar{\mathcal{G}_1} \cap \bar{\mathcal{G}_2}}$$
both in $L_2$ and almost surely.
\end{theorem}

Before proving this theorem, we state and prove a quick lemma.

\begin{lemma}
Let $X$ be square-integrable. Then
$$\lim_{n \rightarrow \infty} \left(S_n(X) - S_{n+1}(X) \right) = 0$$
both in $L_2$ and almost surely. Furthermore, $\sup_{n \geq 1}|S_n(X) - S_{n+1}(X)| \in L_2$ and has norm less than or equal to $||X||$.
\end{lemma}

\begin{proof}
Suppose that $X \in L_2(\Omega, \mathcal{F}, \mathbb{P})$. Then, we have that
\begin{eqnarray*}
||S_n(X) - S_{n+1}(X)||^2 & = & \langle S_n(X) - S_{n+1}(X), S_n(X) - S_{n+1}(X) \rangle\\
 & = & \langle S_{n}(X), S_{n}(X) \rangle - 2 \langle S_{n+1}(X), S_{n}(X) \rangle + \langle S_{n+1}(X), S_{n+1}(X) \rangle\\
 & = & ||S_{n}(X)||^2 - 2 \langle S_{n+1}(X), S_{n}(X) \rangle + ||S_{n+1}(X)||^2\\
 & = & ||S_{n}(X)||^2 - 2 \langle T_{n+1} \cdots T_1(X), T_{n} \cdots T_1(X) \rangle + ||S_{n+1}(X)||^2\\
 & = & ||S_{n}(X)||^2 - 2 \langle T_{n+1} T_{n+1} \cdots T_1(X), T_{n} \cdots T_1(X) \rangle + ||S_{n+1}(X)||^2\\
 & = & ||S_{n}(X)||^2 - 2 \langle T_{n+1} \cdots T_1(X), T_{n+1} T_{n} \cdots T_1(X) \rangle + ||S_{n+1}(X)||^2\\
 & = & ||S_{n}(X)||^2 - ||S_{n+1}(X)||^2
\end{eqnarray*}
Where we used the fact that $T_{n+1}$ is self-adjoint and idempotent since it is a projection operator. Then we will show that
$$\sum_{n=1}^{\infty} |S_n(X(\omega)) - S_{n+1}(X(\omega))|^2 < \infty$$
almost surely. Thus for almost every $\omega \in \Omega$, we have that $\lim_{n \rightarrow \infty} \left(S_n(X(\omega)) - S_{n+1}(X(\omega)) \right) = 0$. To see this we use the monotone convergence theorem in the following computation
\begin{eqnarray*}
\expect{\sum_{n=1}^{\infty} |S_n(X) - S_{n+1}(X)|^2} & = & \expect{\lim_{N \rightarrow \infty}\sum_{n=1}^{N} |S_n(X) - S_{n+1}(X)|^2}\\
 & = & \lim_{N \rightarrow \infty} \expect{\sum_{n=1}^{N} |S_n(X) - S_{n+1}(X)|^2}\\
 & = & \sum_{n=1}^{\infty} \expect{|S_n(X) - S_{n+1}(X)|^2}\\
 & = & \sum_{n=1}^{\infty} ||S_n(X) - S_{n+1}(X)||^2\\
 & = & ||S_1(X)||^2\\
 & = & ||T(X)||^2\\
 & \leq & ||X||^2 < \infty\\
\end{eqnarray*}
In particular we also get that $\expect{|S_n(X) - S_{n+1}(X)|^2} \rightarrow 0$. For the last result, we note that
$$\sup_{n \geq 1}|S_n(X) - S_{n+1}(X)| \leq \left(\sum_{n=1}^{\infty} |S_n(X(\omega)) - S_{n+1}(X(\omega))|^2 \right)^\frac{1}{2}$$
from which it follows that $\sup_{n \geq 1}|S_n(X) - S_{n+1}(X)| \in L_2$ and has norm less than or equal to $||X||$.

\end{proof}

\begin{proof}[Proof of Theorem 4.1]
Let us define $T = T_1 T_2 T_1$. Then $T$ satisfies the conditions of Theorem 3.5. We also have that $T^{2n} = S_{4n+1}$. So we have that $\lim_{n \rightarrow \infty} S_{4n+1}(X) = Q(X)$ almost surely and that $\sup_{n \geq 1}|S_{4n+1}(X)| \in L_2$. Combining this with the previous lemma, we get that $\lim_{n \rightarrow \infty} S_{n}(X) = Q(X)$ both in $L_2$ and almost surely. 

To see that $Q = \expect{\cdotp|\bar{\mathcal{G}_1} \cap \bar{\mathcal{G}_2}}$ we note that by Lemma 3.3, $Q$ is an orthogonal projection. Furthermore, by definition of $T_1$ and $T_2$, its range is $L_2(\Omega, \bar{\mathcal{G}_1} \cap \bar{\mathcal{G}_2}, \mathbb{P})$. However, $\expect{\cdotp|\bar{\mathcal{G}_1} \cap \bar{\mathcal{G}_2}}$ is also an orthogonal projection onto the same subspace. Thus $Q$ is the same operator as $\expect{\cdotp|\bar{\mathcal{G}_1} \cap \bar{\mathcal{G}_2}}$ on $L_2(\Omega, \mathcal{F}, \mathbb{P})$. 
\end{proof}

\section{Application to Sufficiency}

Once again, we fix our measurable space $(\Omega, \mathcal{F})$, and we suppose that we have a collection of probability measures $\{\mathbb{P}_\gamma\}$ for $\gamma \in \Gamma$. We let $\mathcal{N}$ denote the smallest $\sigma$-field containing the sets $N$ that satisfy $\mathbb{P}_\gamma(N) = 0$ for all $\gamma \in \Gamma$. Let us also denote $\mathbb{E}_{\gamma}$ to be expectation with respect to the probability measure $\{\mathbb{P}_\gamma\}$.

\begin{definition}
A $\sigma$-field $\mathcal{G} \subset \mathcal{F}$ is called a sufficient subfield if for each bounded $\mathcal{F}$-measurable function $f$, there is a $\mathcal{G}$-measurable function $g$ that satisfies
$$\mathbb{E}_{\gamma}\left[f \indicator{G} \right] = \mathbb{E}_{\gamma}\left[g \indicator{G} \right]$$
for all $\gamma \in \Gamma$ and for all $G \in \mathcal{G}$. In other words, $\mathcal{G} \subset \mathcal{F}$ is a sufficient subfield if there exists a $g$ such that $g = \mathbb{E}_{\gamma}\left[f | \mathcal{G} \right]$ for all $\gamma \in \Gamma$.

\end{definition}

We consider the following question. If $\mathcal{G}_1, \mathcal{G}_2, ...$ are sufficient subfields of $\mathcal{F}$, which of the following are also sufficient: $\mathcal{G}_1 \cap \mathcal{G}_2$, $\sigma(\mathcal{G}_1 \cup \mathcal{G}_2)$, $\cap_{n=1}^{\infty} \mathcal{G}_n$, $\sigma(\cup_{n=1}^{\infty} \mathcal{G}_n)$. We will show that $\sigma(\mathcal{G}_1 \cap \mathcal{G}_2)$ is not sufficient in general. 

\begin{example}
There is an example where $\sigma(\mathcal{G}_1 \cup \mathcal{G}_2)$ is not sufficient, even though $\mathcal{G}_1$ and $\mathcal{G}_2$ are.
\end{example}

\begin{proof}
Define $\Omega = \{x = (x_1, x_2) \in \mathbb{R}^2 : |x_1| = |x_2| > 0 \}$. Let $\mathbb{P}_x$ be the probability measure that assigns mass $\frac{1}{4}$ to each of the points $(x_1, x_2)$, $(-x_1, x_2)$, $(x_1, -x_2)$, $(-x_1, -x_2)$. We will work with the family of probability measures $\{\mathbb{P}_x\}_{x \in \Omega}$.

In this case, we have that $\mathcal{N} = \{\emptyset, \Omega\}$ since there are no non-trival $\mathbb{P}_x$-null sets. This means that in particular $\mathcal{N}$ is a subfield of any subfield of $\mathcal{F}$. 

Let us define the reflection operators $R_1(X) = (x_1, -x_2)$ and $R_2(X) = (-x_1, x_2)$, and then define the sets $A_{ix} = \{x, R_i(X)\}$ for $i = 1, 2$. Define $\mathcal{G}_i$ to be the smallest $\sigma$-field containing $\{A_{ix} : x \in \Omega\}$ for each $i = 1, 2$. Note that the set $G$ is in $\mathcal{G}_1$ if and only if there is a countable $C \subset \Omega$ such that either $G = \cup_{x \in C} \{A_{1x}\}$ or $G = \left(\cup_{x \in C} \{A_{1x}\}\right)^c$. Similarly $G$ is in $\mathcal{G}_2$ if and only if there is a countable $C \subset \Omega$ such that either $G = \cup_{x \in C} \{A_{2x}\}$ or $G = \left(\cup_{x \in C} \{A_{2x}\}\right)^c$. 

Let us define $\mathcal{G} = \sigma(\mathcal{G}_1 \cup \mathcal{G}_2)$, and also define $\mathcal{F}$ to be the smallest $\sigma$-field containing $\mathcal{G}$ and the set $D = \{x : x \in \Omega, x_1 = x_1\}$. We can also see that 
$$\mathcal{F} = \{(G_1 \cap D) \cup (G_2 \cap D) : G_1, G_2 \in \mathcal{G}\}$$
Thus we have fully specified the probability space we are working with along with a family of probability distributions as well as the subfields in consideration.

We first claim that $\mathcal{G} = \sigma(\mathcal{G}_1 \cup \mathcal{G}_2)$ is not sufficient. Suppose it were suffcient. Then there would exist a $\mathcal{G}$-measurable function $g$ satisfying $\mathbb{P}_x(D \cap G) = \mathbb{E}_{x}\left[g \indicator{G}\right]$ for all $G \in \mathcal{G}$ and all $x \in \Omega$. If we consider the special case of this statement when $G = \{x\}$, then we get that $\mathbb{P}_x(D \cap \{x\}) = \mathbb{E}_{x}\left[g \indicator{\{x\}}\right]$. This statement tells us that $g$ is the characteristic function of the set $D$. However, this is impossible since $D$ is an element of the subfield $\mathcal{G}$. The reason why $D$ cannot be in $\mathcal{G}$ is that all elements of $G$ are either countable or have countable complements as argued earlier. However, both $D$ and $D^c$ are uncountable by construction. So we get a contradiction, and conclude that the subfield $\mathcal{G} = \sigma(\mathcal{G}_1 \cup \mathcal{G}_2)$ is not sufficient.

Our second claim is that both $\mathcal{G}_1$ and $\mathcal{G}_2$ are sufficient subfields of $\mathcal{F}$. We will actually just show that $\mathcal{G}_1$ is sufficient, and then conclude by symmetry that $\mathcal{G}_2$ must also be sufficient. Let $G_1$ and $G_2$ be arbitrary sets in $\mathcal{G}$. Define the functions
\begin{eqnarray*}
f_1 & = & \indicator{G_1 \cap D}\\
f_2 & = & \indicator{G_2 \cap D^2}\\
f & = & f_1 + f_2\\
g_1 & = & f_1 + f_1(R_1(\cdotp))\\
g_2 & = & f_2 + f_2(R_1(\cdotp))\\
g & = & \frac{1}{2}\left(g_1 + g_2\right)
\end{eqnarray*}
We claim that both $g_1$ and $g_2$ are $\mathcal{G}_1$-measurable. We know that $G_1$ is either countable or cocountable. In the case that $G_1$ is countable, then the set $\{x : g_1(X) \neq 0\}$ is countable, and in the case that it is cocountable then the set $\{x : g_1(X) \neq 1\}$ is countable. If we consider an arbitrary Borel set $B \in \mathcal{B}_{\mathbb{R}}$, then we have four possibilities: either none of 0 and 1 are in $B$, both of them are in $B$, 0 is in $B$ but 1 is not, or 1 is in $B$ but 0 is not. In the first case, the preimage of $B$ under $g_1$ is the empty set, and in the second case the preimage is $\Omega$ which are both trivially contained in $\mathcal{G}_1$. In the third case when $0 \in B$ and $1 \notin B$, then by the previous discussion on the countability and cocountability of the sets $\{x : g_1(X) \neq 0\}$ and $\{x : g_1(X) \neq 1\}$, we see that the preimage of $B$ under $g_1$ must be contained in $\mathcal{G}_1$. By the exact same argument, we see that the fourth case also follows. By the same argument given, we can see that $g_2$ is also $\mathcal{G}_1$-measurable, therefore $g$ is also $\mathcal{G}_1$-measurable. 

Let $H$ be a set in $\mathcal{G}_1$, and let $\mathbb{P}_x$ be any probability measure in our family of measures. Then, by construction of $\mathcal{G}_1$, we have that
$$\mathbb{E}_x\left[f(R_1(\cdotp)) \indicator{H}\right] = \mathbb{E}_x\left[f(\cdotp) \indicator{H}\right]$$
Thus, we have that 
\begin{eqnarray*}
\mathbb{E}_x\left[g \indicator{H}\right] & = & \mathbb{E}_x\left[\frac{1}{2}\left(g_1 + g_2\right) \indicator{H}\right]\\
 & = & \mathbb{E}_x\left[\frac{1}{2}\left(f + f(R_1)\right) \indicator{H}\right]\\
 & = & \mathbb{E}_x\left[f \indicator{H}\right]
\end{eqnarray*}
Thus we can conclude that $\mathcal{G}_1$ is a sufficient subfield. By symmetry of the definitions, $\mathcal{G}_2$ is also a sufficient subfield. 

\end{proof}

Despite the negative result for unions of sufficient subfields, we use the results about iterates of conditional expectation operators to show that the result for intersections is positive under the very mild condition that at least one of $\mathcal{G}_1$ or $\mathcal{G}_2$ contain the $\{\mathbb{P}_\gamma\}$-null sets. Note that in the above counterexample, this condition was satisfied since we showed that $\mathcal{N} \subset \mathcal{G}_1$ and $\mathcal{N} \subset \mathcal{G}_2$.

\begin{theorem}
Suppose that $\mathcal{N}$ is contained in one of the subfields $\mathcal{G}_1$ or $\mathcal{G}_2$ (or both). Then the subfield $\mathcal{G} = \mathcal{G}_1 \cap \mathcal{G}_2$ is sufficient.
\end{theorem}

\begin{proof}
Once again, we work in the general setting of the probability space $(\Omega, \mathcal{F}, \{\mathbb{P}_\gamma\})$ and we work with arbitrary sufficient subfields of $\mathcal{F}$,  $\mathcal{G}_1$ and $\mathcal{G}_2$. We can suppose without loss of generality that $\mathcal{N} \subset \mathcal{G}_2$. If not we just swap $\mathcal{G}_1$ and $\mathcal{G}_2$. Let us define the sequence of $\sigma$-fields as follows. For $k \geq 1$ define
\begin{eqnarray*}
\mathcal{G}_{2k - 1} & = & \mathcal{G}_1\\
\mathcal{G}_{2k} & = & \mathcal{G}_2
\end{eqnarray*}
Fix a function $f$ that is $\mathcal{F}$-measurable and bounded. By the definition of sufficiency, there exists a function $g_1$ satisfying 
$$g_1 = \mathbb{E}_{\gamma}\left[f | \mathcal{G_1} \right]$$
for all $\gamma \in \Gamma$. Similarly, by use of sufficiency again, we can inductively define $g_n$ to be the function that satisfies
$$g_n = \mathbb{E}_{\gamma}\left[g_{n-1} | \mathcal{G_1} \right]$$
for all $\gamma \in \Gamma$. Let us define the function $g(X)$ and $h(X)$ pointwise as follows
$$
g(X) = \left\{ 
 \begin{array}{ll}
 \lim_{n \rightarrow \infty} g_{2n-1}(X) & \mbox{if the limit exists} \\
 0 & \mbox{otherwise}
 \end{array} \right.
$$
$$
h(X) = \left\{ 
 \begin{array}{ll}
 \lim_{n \rightarrow \infty} g_{2n}(X) & \mbox{if the limit exists} \\
 0 & \mbox{otherwise}
 \end{array} \right.
$$
Thus by applying Theorem 4.1, we see that 
$$\lim_{n \rightarrow \infty} g_n = \expect{f|\bar{\mathcal{G}_1} \cap \bar{\mathcal{G}_2}}$$
where we denote $\bar{\mathcal{G}_i}$ to be the completion of the $\sigma$-field $\mathcal{G}_1$ with respect to the set of probability measures $\{\mathbb{P}_\gamma\}$. Thus we have that $g = \expect{f|\bar{\mathcal{G}_1} \cap \bar{\mathcal{G}_2}}$ and that $\{g(X) \neq h(X)\} \in \mathcal{N}$. The reason why this is true is that the convergence in Theorem 4.1 is almost surely. In particular we have that $g - h$ is $\mathcal{G}_2$-measurable since $\mathcal{G}_2$ contains $\mathcal{N}$. Thus by measurability of $g - h$ and $h$ with respect to $\mathcal{G}_2$, we can conclude that $g$ is also $\mathcal{G}_2$-measurable. By construction, $g$ is also $\mathcal{G}_1$-measurable. Thus $g$ is $\mathcal{G}_1 \cap \mathcal{G}_2$-measurable. 

We then note that $\bar{\mathcal{G}_1} \cap \bar{\mathcal{G}_2} \supset \mathcal{G}_1 \cap \mathcal{G}_2$ and by the tower and idempotence properties of conditional expectations, 
\begin{eqnarray*}
\mathbb{E}_{\gamma}\left[f | \mathcal{G}_1 \cap \mathcal{G}_2\right] & = & \mathbb{E}_{\gamma}\left[\mathbb{E}_{\gamma}\left[f | \bar{\mathcal{G}_1} \cap \bar{\mathcal{G}_2}\right] | \mathcal{G}_1 \cap \mathcal{G}_2\right]\\
 & = & \mathbb{E}_{\gamma}\left[g | \mathcal{G}_1 \cap \mathcal{G}_2\right]\\
 & = & g
\end{eqnarray*}
Thus we see that the subfield $\mathcal{G}_1 \cap \mathcal{G}_2$ satisfies the definition of sufficiency.
\end{proof}

We now consider the remaining two questions concerning the sufficiency of $\cap_{n=1}^{\infty} \mathcal{G}_n$ and $\sigma(\cup_{n=1}^{\infty} \mathcal{G}_n)$. By the counter example given, we conclude that the second subfield may not be sufficient in general, by taking $\mathcal{G}_n = \mathcal{G}_2$ for $n \geq 2$ and then invoking the above example. However, we will answer the question about the sufficiency of $\cap_{n=1}^{\infty} \mathcal{G}_n$ affirmatively.

\begin{theorem}
Let $\{\mathcal{G}_n\}_{n=1}^{\infty}$ be a countable collection of sufficient subfields. Then if $\mathcal{G}_{n} \supset \mathcal{G}_{n+1}$ for all $n \geq 1$, then $\cap_{n=1}^{\infty} \mathcal{G}_n$ is sufficient.
\end{theorem}

\begin{proof}
By definition of sufficiency, if we are given any bounded $\mathcal{F}$-measurable function $f$, there exists a $\mathcal{G}_{n}$-measurable function $g_n$ such that $g_n = \mathbb{E}_{\gamma}\left[f|\mathcal{G}_{n}\right]$ for all $\gamma \in \Gamma$. Define the function $g$ pointwise by
$$
g(X) = \left\{ 
 \begin{array}{ll}
 \lim_{n \rightarrow \infty} g_{n}(X) & \mbox{if the limit exists} \\
 0 & \mbox{otherwise}
 \end{array} \right.
$$

Suppose that we are in the first situation, where we have a decreasing sequence of $\sigma$-fields. Then we have that $g$ is $\cap_{n=1}^{\infty} \mathcal{G}_n$-measurable. By the continuity of conditional expectations, we have that 
$$\mathbb{E}_{\gamma}\left[f | \cap_{n=1}^{\infty} \mathcal{G}_{n}\right] = \lim_{n \rightarrow \infty} g_n$$
Thus $g = \mathbb{E}_{\gamma}\left[f | \cap_{n=1}^{\infty} \mathcal{G}_{n}\right]$ for all $\gamma \in \Gamma$. So we conclude that $\cap_{n=1}^{\infty} \mathcal{G}_n$ is a sufficient subfield.
\end{proof}

We combine the above with the theorem of the sufficiency of the intersection of two subfields to get a proof for the sufficiency of a countable intersection below.

\begin{theorem}
Let $\{\mathcal{G}_n\}_{n=1}^{\infty}$ be a countable collection of sufficient subfields such that $\mathcal{N} \subset \mathcal{G}_n$ for all $n \geq 1$. Then $\cap_{n=1}^{\infty} \mathcal{G}_n$ is also sufficient.
\end{theorem}

\begin{proof}
Define $\mathcal{H}_n = \cap_{k=1}^{n} \mathcal{G}_k$. Then, by applying induction to Theorem 5.1, we see that for each $n \geq 1$, $\mathcal{H}_n$ is sufficient. Since $\mathcal{H}_{n} \supset \mathcal{H}_{n+1}$, we conclude by Theorem 5.2 that $\cap_{n=1}^{\infty} \mathcal{H}_n$ is also sufficient. But $\cap_{n=1}^{\infty} \mathcal{G}_n = \cap_{n=1}^{\infty} \mathcal{H}_n$, so $\cap_{n=1}^{\infty} \mathcal{G}_n$ is a sufficient subfield.  
\end{proof}

\bibliographystyle{plain}
\bibliography{biblio}

\end{document}